\documentclass[10pt]{article}
\usepackage[dvips]{graphicx}
\usepackage{amssymb,color}
\usepackage[latin1]{inputenc}
\usepackage{epic}
\usepackage{pstricks}
\usepackage{pst-node}
\usepackage{pstricks-add}
\usepackage[latin1]{inputenc}
\usepackage{color}
\usepackage{lettrine}
\usepackage{calc,pifont}
\usepackage[noindentafter,calcwidth]{titlesec}
\usepackage{pstricks}
\usepackage{pst-plot}
\usepackage{pst-node}
\usepackage{amsmath}
\usepackage{amssymb}
\usepackage{amsthm}
\usepackage{subfigure}
\usepackage{lineno}

\def\BBox{\kern  -0.2cm\hbox{\vrule width 0.2cm height 0.2cm}}

\newtheorem{example}{Example}

\newtheorem{teo}{Theorem}[section]

\theoremstyle{definition}

\theoremstyle{remark}

\title{A note on strong Skolem starters}
\author{Adrián Vázquez-Ávila\thanks{adrian.vazquez@unaq.edu.mx}\\
{\small Subdirección de Ingeniería y Posgrado}\\
{\small Universidad Aeronáutica en Querétaro}\\
}

\date{}

\begin{document}
%\linenumbers
\maketitle
%%%%%%%%%%%%%%%%%%%%%%%%%%%%%%%%%%%%%%%%%%%%%%
\begin{abstract}
%Let $\mathbb{Z}_n$ be a the finite additive group of integers
%modulo $n$ of odd order $n$, and let $\mathbb{Z}^*_n=\mathbb{Z}_n\setminus\{0\}$ be the set of non-zero elements. A \emph{starter} for $\mathbb{Z}_n$ is a set $S=\{\{x_i,y_i\}\}_{i=1}^{\frac{n-1}{2}}$ such that $\bigcup_{i=1}^{\frac{n-1}{2}}\{x_i,y_i\}=G^*$ and $\{\pm(x_i-y_i)\}_{i=1}^{\frac{n-1}{2}}=\mathbb{Z}^*_n$. Moreover, if $|\left\{\{x_i+y_i\}\right\}_{i=1}^{\frac{n-1}{2}}|=\frac{n-1}{2}$, then $S$ is called a \emph{strong starter} for $\mathbb{Z}_n$.
%Let $n=2q+1$, and $1<2<\ldots<2q$ be the order of the non-zero integers modulo $n$. A starter for $\mathbb{Z}_n$ is Skolem if it can be written as $S=\{\{x_i,y_i\}\}_{i=1}^{\frac{n-1}{2}}$ such that $y_i>x_i$ and $y_i-x_i=i$, for $i=1,\ldots,q$.

In 1991, Shalaby conjectured that any additive group $\mathbb{Z}_n$, where $n\equiv1$ or 3
(mod 8) and $n \geq11$, admits a strong Skolem starter and constructed these starters of all admissible orders $11\leq n\leq57$. Only finitely many strong Skolem starters have been known. Recently, in [O. Ogandzhanyants, M. Kondratieva and N. Shalaby, \emph{Strong Skolem Starters}, J. Combin. Des. {\bf 27} (2018), no. 1, 5--21] was given an infinite families of them. In this note, an infinite family of strong Skolem starters for $\mathbb{Z}_n$, where $n\equiv3$ mod 8 is a prime integer, is presented.
\end{abstract}

%%%%%%%%%%%%%%%%%%%%%%%%%%%%%%%%%%%%%%%%%%%%%%

%{\bf MSC 2000.} ~05C35.
%%%%%%%%%%%%%%%%%%%%%%%%%%%%%%%%%%%%%%%%%%%%%%%%%%%%%%
%%%%%%%%%%%%%%%%%%%%%%%%%%%%%%%%%%%%%%%%%%%%%%%%%%%%%%INTRODUCTION
\section{Introduction}
Let $G$ be a finite additive abelian group of odd order $n=2q+1$, and let $G^*=G\setminus\{0\}$ be the set of non-zero elements of $G$. A \emph{starter} for $G$ is a set $S=\{\{x_1,y_1\},\ldots,\{x_q,y_q\}\}$ such that
$\left\{x_1,\ldots,x_q,y_1,\ldots,y_q\right\}=G^*$
and $\left\{\pm(x_i-y_i):i=1,\ldots,q\right\}=G^*$. Moreover, if $\left|\left\{x_i+y_i:i=1,\ldots,q\right\}\right|=q$, then $S$ is called \emph{strong starter} for $G$.	
%It is no hard to see that, if $S=\left\{\{x,-x\}: x\in G^*\right\}$ then $S$ is a starter for G.

Strong starters were first introduced by Mullin and Stanton in \cite{MR0234587} in constructing of Room squares. Starters and strong starters have been useful to construct many combinatorial designs such as Room cubes \cite{MR633117}, Howell designs \cite{MR728501,MR808085}, Kirkman triple systems \cite{MR808085,MR0314644}, Kirkman squares and cubes \cite{MR833796,MR793636},  and factorizations of complete graphs \cite{MR0364013,MR2206402,MR1010576,MR623318,MR685627}. Moreover, there are some interesting results on strong starters for cyclic groups \cite{Avila,MR808085}, in particular for $\mathbb{F}_q$ \cite{Avila,MR0325419,dinitz1984,MR0392622,MR0249314,MR0260604}, and for finite abelian groups \cite{MR1010576,MR1044227}.

The Skolem starters, which we will deal, are defined for additive groups $\mathbb{Z}_n$ of integers modulo $n$.

Let $n=2q+1$, and $1<2<\ldots<2q$ be the order of $\mathbb{Z}_n^*$. A starter for $\mathbb{Z}_n$ is \emph{Skolem} if it can be written as $S=\{\{x_i,y_i\}\}_{i=1}^q$ such that $y_i>x_i$ and $y_i-x_i=i$ (mod n), for $i=1,\ldots,q$. In \cite{ShalabyThesis}, it was proved that, the Skolem starter for $\mathbb{Z}_n$ exits if and only if $n\equiv1$ or 3 mod 8.

A starter which is both Skolem and strong is called \emph{strong Skolem starter}.

\begin{example}
The set $S=\{\{9,10\},\{2,4\},\{5,8\},\{3,7\},\{1,6\}\}$ is a strong Skolem starter for $\mathbb{Z}_{11}$.
\end{example}

In \cite{ShalabyThesis}, it was conjectured that any additive group $\mathbb{Z}_n$, where $n\equiv1$ or 3
mod 8 and $n \geq11$, admits a strong Skolem starter and constructed these starters of all admissible orders $11\leq n\leq57$. Only finitely many strong Skolem starters have been known. Recently, in \cite{Shalaby}, it was given a geometrical interpretation of strong Skolem starters and it was constructed an infinite families of them. In this note, an infinite family of strong Skolem starters for $\mathbb{Z}_n$, where $n\equiv3$ mod 8 is a prime integer, is presented.

This paper is organized as follows. 
In Section 2, we recall some basic properties about quadratic residues. And, in section 3, we prove the main theorem and we present some examples.
%%%%%%%%%%%%%%%%%%%%%%%%
%%%%%%%%%%%%%%%%%%%%%%%%
%%%%%%%%%%%%%%%%%%%%%%%%
%%%%%%%%%%%%%%%%%%%%%%%%
%%%%%%%%%%%%%%%%%%%%%%%%
%%%%%%%%%%%%%%%%%%%%%%%%
\section{Quadratic residues}\label{sec:quadratic}
Let $q$ be an odd prime power. An element $x\in\mathbb{F}_q^*$ is called a \emph{quadratic residue} if there exists an element $y\in\mathbb{F}_q^{*}$ such that $y^2=x$. If there is no such $y$, then $x$ is called a \emph{non-quadratic residue.} The set of quadratic residues of $\mathbb{F}_q^{*}$ is denoted by $QR(q)$ and the set of non-quadratic residues is denoted by $NQR(q)$. It is well known that $QR(q)$ is a cyclic subgroup of $\mathbb{F}_q^{*}$ of cardinality $\frac{q-1}{2}$ (see for example \cite{MR2445243}). As well as, it is well known, if either $x,y\in QR(q)$ or $x,y\in NQR(q)$, then $xy\in QR(q)$, and if $x\in QR(q)$ and $y\in NQR(q)$, then $xy\in NQR(q)$.

The following theorems are well known results on quadratic residues. 
For more details of this kind of results the reader may consult \cite{burton2007elementary,MR2445243}.

\begin{teo}[Eulers' criterion]
	Let $q$ be an odd prime and $x\in\mathbb{F}_q^{*}$, then
	\begin{enumerate}
		\item $x\in QR(q)$ if and only if $x^{\frac{q-1}{2}}=1$.
		\item $x\in NQR(q)$ if and only if $x^{\frac{q-1}{2}}=-1$.
	\end{enumerate}
\end{teo}

\begin{teo}\label{col:menosuno}
	Let $q$ be an odd prime power, then
	\begin{enumerate}
		\item $-1\in QR(q)$ if and only if $q\equiv1$ mod $4$.
		\item $-1\in NQR(q)$ if and only if $q\equiv3$ mod $4$.
	\end{enumerate}
\end{teo}

\begin{teo}\label{inverso}
	Let q be an  odd prime. If $q\equiv3$ mod $4$, then 
	\begin{enumerate}
		\item $x\in QR$ if and only if $-x\in QNR$.
		\item $x\in NQR$ if and only if $-x\in QR$.
	\end{enumerate}
\end{teo}
%%%%%%%%%%%%%%%%%%%%%%%%
%%%%%%%%%%%%%%%%%%%%%%%%
%%%%%%%%%%%%%%%%%%%%%%%%
%%%%%%%%%%%%%%%%%%%%%%%%
%%%%%%%%%%%%%%%%%%%%%%%%
%%%%%%%%%%%%%%%%%%%%%%%%
\section{Main results}\label{sec:main}
In \cite{Avila}, it was proved that, if $q\equiv3$ mod 4 is an odd prime power with $q\neq3$, and $\alpha$ is a generator of $QR(q)$ and $\beta+1\neq0$, then the following set
\begin{eqnarray*}\label{strong_1}
S_\beta=\left\{\{\alpha,\alpha\beta\},\{\alpha^2,\alpha^2\beta\},\ldots,\{\alpha^\frac{q-1}{2},\alpha^\frac{q-1}{2}\beta\}\right\},
\end{eqnarray*} 
is a strong starter for $\mathbb{F}_q$. 

Hence, we have the following

\begin{teo}[Main Theorem]
If $q\equiv3$ (mod 8) is an odd prime, $\alpha\in\mathbb{Z}_q$ is a generator of $QR(q)\subseteq\mathbb{Z}_q$ and $\beta\in\{2,\frac{q+1}{2}\}$, then the strong starter $$S_\beta=\left\{\{\alpha,\beta\alpha\},\{\alpha^2,\beta\alpha^2\},\ldots,\{\alpha^\frac{q-1}{2},\beta\alpha^\frac{q-1}{2}\}\right\},$$ is Skolem for $\mathbb{Z}_q$.
\end{teo}

\begin{proof}
Let $q=2t+1$ be an odd prime such that $q\equiv3$ (mod 8), and let $1<2<\ldots<2t$ be the order of the non-zero elements of $\mathbb{Z}_q^*$. Since $q\equiv3$ (mod 8) then $2\in NQR(q)$ (see for example \cite{Ireland}), which implies that $\frac{q+1}{2}=2^{-1}\in NQR(q)$, since $2\cdot2^{-1}\in QR(q)$. Let $Q_{\frac{1}{2}}=\{1,2,\ldots,t\}$, $\beta\in\{2,\frac{q+1}{2}\}$ and  $$S_\beta=\left\{\{\alpha,\beta\alpha\},\{\alpha^2,\beta\alpha^2\},\ldots,\{\alpha^\frac{q-1}{2},\beta\alpha^\frac{q-1}{2}\}\right\}.$$

We will prove that, if $\beta\alpha^i>\alpha^i$ then $\beta\alpha^i-\alpha^i\in Q_\frac{1}{2}$, and if $\alpha^i>\beta\alpha^i$ then $\alpha^i-\beta\alpha^i\in Q_\frac{1}{2}$, for all $i=1,\ldots,t$.	
\begin{itemize}
	\item [case(i):] Let us suppose that $\beta=2$. If $\alpha^i\in Q_\frac{1}{2}$, for some $i\in\{1,\ldots,t\}$, then $2\alpha^i>\alpha^i$, which implies that $2\alpha^i-\alpha^i=\alpha^i\in Q_\frac{1}{2}$. On the other hand, if $\alpha^i\not\in Q_\frac{1}{2}$, for some $i\in\{1,\ldots,t\}$, then $-\alpha^i\in Q_{\frac{1}{2}}$. Hence, $2(-\alpha^i)>-\alpha^i$, which implies that $-2\alpha^i+\alpha^i=-\alpha^i\in Q_{\frac{1}{2}}$. Therefore, $S_2$ is a strong Skolem starter for $\mathbb{Z}_q$. 
	
	\item [case (ii):] Let us suppose that $\beta=\frac{q+1}{2}=2^{-1}$. If $\beta\alpha^i\in Q_\frac{1}{2}$, for some $i\in\{1,\ldots,t\}$, then  $2(\beta\alpha^i)=\alpha^i>\beta\alpha^i$, which implies that $\alpha^i-\beta\alpha^i=\beta\alpha^i\in Q_\frac{1}{2}$, since $2\beta=1$. On the other hand, if $\beta\alpha^i\not\in Q_\frac{1}{2}$, for some $i\in\{1,\ldots,t\}$, then $-\beta\alpha^i\in Q_{\frac{1}{2}}$. Hence, $2(-\beta\alpha^i)=-\alpha^i>-\beta\alpha^i$, which implies that $-\alpha^i+\beta\alpha^i=-\beta\alpha^i\in Q_{\frac{1}{2}}$. Therefore, $S_{2^{-1}}$ is a strong Skolem starter for $\mathbb{Z}_q$.\qed 
\end{itemize}
\end{proof}

%%%%%%%%%%%%%%%%%%%%%%%%
%%%%%%%%%%%%%%%%%%%%%%%%
%%%%%%%%%%%%%%%%%%%%%%%%
%%%%%%%%%%%%%%%%%%%%%%%%
%%%%%%%%%%%%%%%%%%%%%%%%
%%%%%%%%%%%%%%%%%%%%%%%%
%%%%%%%%%%%%%%%%%%%%%%%%
%%%%%%%%%%%%%%%%%%%%%%%%
%%%%%%%%%%%%%%%%%%%%%%%%
%%%%%%%%%%%%%%%%%%%%%%%%
%%%%%%%%%%%%%%%%%%%%%%%%
%%%%%%%%%%%%%%%%%%%%%%%%
\subsection{Examples}
In this subsection there will be provided examples of strong Skolem starters for $\mathbb{Z}_{11}$, $\mathbb{Z}_{19}$ and $\mathbb{Z}_{43}$ given by The Main Theorem.

Let consider $\mathbb{Z}_{11}$ then $\alpha=4$ is a generator of $QR(11)$. Hence
\begin{center}
$QR(11)=\{1,3,4,5,9\}$ and $NQR(11)=\{2,6,7,8,10\}$.
\end{center}
Therefore $$S_2=\left\{\{4,8\},\{5,10\},\{9,7\},\{3,6\},\{1,2\} \right\}=\left\{\{1,2\},\{7,9\},\{3,6\},\{4,8\},\{5,10\}\right\},$$
and
$$S_6=\left\{\{4,2\},\{5,8\},\{9,10\},\{3,7\},\{1,6\} \right\}=\left\{\{9,10\},\{2,4\},\{5,8\},\{3,7\},\{1,6\}\right\},$$are strong Skolem starters for $\mathbb{Z}_{11}$.

Onthe other hand, let consider $\mathbb{Z}_{19}$ then $\alpha=4$ is a generator of $QR(19)$. Hence
$QR(19)=\{1,4,5,6,7,9,11,16,17\}$ and $NQR(19)=\{2,3,8,10,12,13,15,18\}$.

Therefore
\begin{eqnarray*}
	S_2&=&\{\{4,8\},\{16,13\},\{7,14\},\{9,18\},\{17,15\},\{11,3\},\{6,12\},\{5,10\},\{1,2\}\}\\
	&=&\{\{1,2\}, \{15,17\}, \{13,16\}, \{4,8\}, \{5,10\}, \{6,12\}, \{7,14\}, \{3,11\}, \{9,18\}\},
\end{eqnarray*}

and

\begin{eqnarray*}
	S_{10}&=&\{\{4,2\},\{16,8\},\{7,13\},\{9,14\},\{17,18\},\{11,15\},\{6,3\},\{5,12\},\{1,10\}\}\\
	&=&\{\{17,18\},\{2,4\},\{3,6\},\{11,15\},\{9,14\}, \{7,13\},\{5,12\},\{8,16\},\{1,10\}\},
\end{eqnarray*}
are strong Skolem starters for $\mathbb{Z}_{19}$.

Finally, let consider $\mathbb{Z}_{43}$ then $\alpha=9$ is a generator of $QR(43)$. Hence
\begin{center}
$QR(43)=\{1,4,6,9,10,11,13,14,15,16,17,21,23,24,25,31,35,36,38,40,41\}$\\ and \\$NQR(43)=\{2,3,5,7,8,12,18,19,20,22,26,27,28,29,30,32,33,34,37,39,42\}$.
\end{center}
Therefore
\begin{eqnarray*}
	S_2&=&\{\{9,18\},\{38,33\},\{41,39\},\{25,7\},\{10,20\},\{4,8\},\{36,29\},\\
	&&\{23,3\},\{35,27\},\{14,28\},\{40,37\},\{16,32\},\{15,30\},\{6,12\},\\
	&&\{11,22\},\{13,26\},\{31,19\},\{21,42\},\{17,34\},\{24,5\},\{1,2\}\},\\
	&=&\{\{1,2\},\{39,41\},\{37,40\},\{4,8\},\{33,38\},\{6,12\},\{29,36\},\\
	&&\{27,35\},\{9,18\},\{10,20\},\{11,22\},\{19,31\}  ,\{13,26\},\{14,28\},\\
	&&\{15,30\},\{16,32\},\{17,34\},\{7,25\},\{5,24\},\{3,23\},\{21,42\}\}
\end{eqnarray*}
and
\begin{eqnarray*}
	S_{22}&=&\{\{9,26\},\{38,19\},\{41,42\},\{25,34\},\{10,5\},\{4,2\},\{36,18\},\\
	&&\{23,33\},\{35,39\},\{14,7\},\{40,20\},\{16,8\},\{15,29\},\{6,3\},\\
	&&\{11,27\},\{13,28\},\{31,37\},\{21,32\},\{17,30\},\{24,12\},\{1,22\}\},\\
	&=&\{\{41,42\},\{2,4\},\{3,6\},\{35,39\},\{5,10\},\{31,37\},\{7,14\},\\
	&&\{8,16\},\{25,34\},\{23,33\},\{21,32\},\{12,24\},\{17,30\},\{15,29\},\\
	&&\{13,28\},\{11,27\},\{9,26\},\{18,36\},\{19,38\},\{20,40\},\{1,22\}\},
\end{eqnarray*}
are strong Skolem starters for $\mathbb{Z}_{43}$.

%%%%%%%%%%%%%%%%%%%%%%%%
%%%%%%%%%%%%%%%%%%%%%%%%
%%%%%%%%%%%%%%%%%%%%%%%%
%%%%%%%%%%%%%%%%%%%%%%%%
%%%%%%%%%%%%%%%%%%%%%%%%
%%%%%%%%%%%%%%%%%%%%%%%%

\

{\bf Acknowledgment}

Research was partially supported by SNI and CONACyT.

\bibliographystyle{amsplain}

\end{document}